\documentclass[12pt]{article}
\usepackage[utf8]{inputenc}
\usepackage[T1]{fontenc}

\usepackage{amsmath}
\usepackage{amsfonts}
\usepackage{latexsym}
\usepackage[final]{graphicx}  % with [final], images appear in draft mode
\usepackage{amssymb}
\usepackage{amsthm}
\usepackage[margin=2cm]{geometry}
\usepackage{url}
\usepackage{color}
\usepackage{enumerate}
\usepackage[shortlabels]{enumitem} %pour commencer les enumerations a des nombres differents
\usepackage[small]{caption}
\usepackage[noadjust]{cite}
\usepackage{framed}
\usepackage{cleveref}
\usepackage{ifdraft}
\usepackage[affil-it]{authblk}
\usepackage{tikz}
\usetikzlibrary{arrows,shapes}
\usepackage{caption}
\usepackage{subcaption}

\ifdraft{
    \usepackage[colorinlistoftodos]{todonotes}
    \usepackage{xcolor}
    \usepackage[mark]{gitinfo2}
    
}{
    \usepackage[disable]{todonotes}
}

\newif\iflongversion
\longversiontrue    % uncomment this line for the long version
\iflongversion
  \newenvironment{longversion}{}{}
  \usepackage{comment}
  \excludecomment{extendedabstractversion}
\else
  \newenvironment{extendedabstractversion}{}{}
  \usepackage{comment}
  \excludecomment{longversion}
  \usepackage{environ}
  \NewEnviron{killcontents}{}

\fi
\newtheorem{definition}{Definition}
\newtheorem{lemma}[definition]{Lemma}
\newtheorem{proposition}[definition]{Proposition}
\newtheorem{corollary}[definition]{Corollary}

\newtheorem*{conjecture*}{Conjecture}
\newtheorem{theorem}[definition]{Theorem}
\theoremstyle{remark}
\newtheorem{example}[definition]{Example}
\newtheorem{remark}[definition]{Remark}

\newcommand{\N}{\mathbb{N}}

\newcommand{\til}{\widetilde}

\newcommand{\Lst}{\textsc{Lst}}
\newcommand{\Fst}{\textsc{Fst}}

\newcommand{\A}{\mathcal A}

\newcommand{\C}{\mathcal C}
\newcommand{\bu}{{\bf u}}
\newcommand{\bv}{{\bf v}}

\renewcommand{\L}{\mathcal L}
\newcommand{\Lu}{\L(\bu)}

\renewcommand{\P}{\mathcal P}

\crefname{theorem}{theorem}{theorems}
\crefname{corollary}{corollary}{corollaries}
\crefname{example}{example}{examples}
\crefname{lemma}{lemma}{lemmas}
\crefname{proposition}{proposition}{propositions}
\crefname{definition}{definition}{definitions}

\begin{document}

\title{On the Zero Defect Conjecture}
    % \thanks{European Journal of Combinatorics (2016) 200--214.
    % \tt{http://dx.doi.org/10.1016/j.ejc.2015.05.006}}

% \author{{\sc S\'ebastien Labb\'e
% \quad and \quad
% Edita Pelantov\'a \quad and \quad \v{S}t\v{e}p\'an Starosta}\\  \\
% \small LIAFA, Universit\'e Paris Diderot - Paris 7,\\ [-0.6ex]
% \small Case 7014, 75205 Paris Cedex 13, France\\ [-0.6ex]
% \small \tt labbe@liafa.univ-paris-diderot.fr\\ [.9ex]
% \small Czech Technical University in Prague\\[-0.6ex]
% \small Trojanova 13, 120 00 Praha 2, Czech Republic\\[-0.6ex]
% \small \tt edita.pelantova@fjfi.cvut.cz}

\author{S\'ebastien Labb\'e%
    \thanks{Electronic address: \texttt{slabbe@ulg.ac.be}}}
\affil{\footnotesize Universit\'e de Li\`ege\\
B\^at. B37 Institut de Math\'ematiques\\
Grande Traverse 12, 4000 Li\`ege, Belgium}

\author{Edita Pelantová}
\affil{\footnotesize Faculty of Nuclear Sciences and Physical Engineering\\ Czech Technical University in Prague\\ Czech Republic}

\author{Štěpán Starosta%
  \thanks{Electronic address: \texttt{stepan.starosta@fit.cvut.cz}}}
\affil{\footnotesize Faculty of Information Technology\\ Czech Technical University in Prague\\ Czech Republic}

\date{}

\maketitle

\begin{abstract}
\noindent Brlek et al. conjectured in 2008 that any fixed point of a primitive morphism with finite palindromic defect  is either periodic or
 its palindromic defect is  zero.  Bucci and Vaslet   disproved this conjecture in  2012 by a counterexample over ternary alphabet. We prove that the conjecture is valid on binary alphabet. We also  describe a class of  morphisms over   multiliteral alphabet  for which the conjecture still holds.
The proof is based on properties of extension graphs.
\end{abstract}

%\ifdraft{
  %  \listoftodos
%}


\section{Introduction}

% Bojan Basic?
%A morphism is \textit{marked} if ....

Palindromes --- words read the same from the left as from the right --- are a favorite pun in various languages.
For instance, the words ressasser, ťahať, and šílíš are palindromic words in the first languages of the authors of this paper.
The reason for a study of palindromes in formal languages is not only to deepen the theory, but it has also applications.

The theoretical reasons include the fact that a Sturmian word, i.e., an infinite aperiodic word with the least factor complexity, can be characterized using the number of palindromic factors of given length that occur in a word, see \cite{DrPi}.
The application motives include the study of the spectra of discrete Schr\"{o}dinger operators, see \cite{HoKnSi,JiSi}.

In \cite{DrJuPi}, the authors provide an elementary observation that a finite word of length $n$ cannot contain more than $n+1$ (distinct) palindromic factors, including the empty word as a palindromic factor.
We illustrate this on the following 2 examples of words of length $9$:
\[
w^{(1)} = 010010100 \quad \text{ and } \quad w^{(2)}  = 011010011.
\]
The word $w^{(1)} $ is a prefix of the famous Fibonacci word and $w_2$ is a prefix of (also famous) Thue--Morse word.
There are $10$ palindromic factors of $w^{(1)}$: $0$, $1$, $00$, $010$, $101$, $1001$, $01010$, $010010$, $0010100$, and the empty word.
The word $w^{(2)} $ contains only $9$ palindromes: $0$, $1$, $11$, $0110$, $101$, $010$, $00$, $1001$, and the empty word.

The existence of the upper bound on the number of distinct palindromic factors lead to the definition of \textit{palindromic defect} (or simply \textit{defect}) of a finite word $w$, see \cite{BrHaNiRe}, as the value
\[
D(w) = n + 1 - \text{ the number of palindromic factors of } w
\]
with $n$ being the length of $w$.
Our examples satisfy $D(w^{(1)}) = 0$, i.e., the upper bound is attained, and
$D(w^{(2)}) = 1$.
The notion of palindromic defect naturally extends to infinite words.
For an infinite word $\bu$ we set
\[
D(\bu) = \sup \{ D(w) \colon w \text{ is a factor of } \bu \}.
\]

In this paper, we deal with infinite words that are generated by a primitive morphism of a free monoid $\A^*$ with $\A$ being a finite alphabet.
A morphism $\varphi$ is completely determined by the images of all letters $a \in \A$: $a \mapsto \varphi(a) \in \A^*$.
A morphism is \textit{primitive} if  there exists a power $k$ such that any letter $b \in \A$ appears in  the word $\varphi^k(a)$ for any letter  $a \in \A$.

The two mentioned infinite words can be generated using a primitive morphism.
Consider the morphism $\varphi_F$ over $\{0, 1\}^*$ determined by $0 \mapsto 01$ and $1 \mapsto 0$.
By repeated application of $\varphi_F$, starting from $0$, we obtain
\[
0 \mapsto 01 \mapsto 010 \mapsto 01001 \mapsto 01001010 \ldots
\]
Since $\varphi_F^n(0)$ is a prefix of $\varphi_F^{n+1}(0)$ for all $n \in \N$, there exists an infinite word $\bu_F$, called the Fibonacci word, such that $\varphi_F^n(0)$ is its prefix for all $n$.
Consider a natural extension of $\varphi_F$ to infinite words, we obtain that $\bu_F$ is a fixed point of $\varphi_F$ since
\[
\bu_F = \varphi_F(\bu_F) = \varphi_F(u_0u_1u_2 \ldots) = \varphi_F(u_0) \varphi_F(u_1) \varphi_F(u_2) \ldots
\]
where $u_i \in \{0,1\}$.

Similarly, let $\varphi_{TM}$ be a morphism determined by $0 \mapsto 01$ and $1 \mapsto 10$.
By repeated application of $\varphi_{TM}$, starting again from $0$, we obtain
\[
0 \mapsto 01 \mapsto 0110 \mapsto 01101001 \mapsto 0110100110010110 \ldots
\]
The infinite word having $\varphi_{TM}^n(0)$ as a prefix for each $n$ is the Thue--Morse word, sometimes also called Prouhet--Thue--Morse word.

The present article focuses on palindromic defect of infinite words which are fixed points of primitive morphisms.
In order for the palindromic defect of such an infinite word to be finite, the word must contain an infinite number of palindromic factors.
This property is satisfied by the two mentioned words $\bu_F$ and $\bu_{TM}$. However, for their palindromic defect, we have $D(\bu_F) = 0$, whilst $D(\bu_{TM}) = + \infty$.

There exist fixed points $\bu$ of primitive morphisms with $0 < D(\bu) < +\infty$,
but on a two-letter alphabet, only ultimately periodic words are known.
 In \cite{BrHaNiRe}, examples of such words are  given by Brlek, Hamel, Nivat and Reutenauer as follows:  for any  $k \in \mathbb{Z}, k\geq 2$  denote by $z$ the
finite word
\[
z = 0 1^k 0 1^{k-1} 0 0 1^{k-1} 0 1^{k} 0\,.
\]
Then the infinite periodic word $z^{\omega}$ has palindromic defect $k$.  Of course,  the periodic  word  $z^{\omega}$ is  fixed by the primitive  morphism  $0 \mapsto z, 1 \mapsto z$.
In \cite{BlBrGaLa}, the authors stated the following conjecture:
\begin{conjecture*}[Zero Defect Conjecture] \label{conj:zerodefectconjecture}
If $\bu$ is a fixed point of a primitive morphism such that $D(\bu) < +
\infty$, then $\bu$ is periodic or $D(\bu) = 0$.
\end{conjecture*}

In 2012, Bucci and Vaslet \cite{BuVa12} found a counterexample to this conjecture on a ternary alphabet.
They showed that the fixed point of the primitive morphism determined by
\[
a \mapsto aabcacba, b \mapsto aa, c \mapsto a
\]
has finite positive palindromic defect and is not periodic.

\begin{longversion}
In this article, we show that the conjecture is valid on a binary alphabet.
Then we generalize the  method used for  morphisms on a binary alphabet to marked morphisms on a multiliteral alphabet.
The main result of the article is the following theorem.
\end{longversion}

\begin{extendedabstractversion}
In this article, we show that the conjecture is valid on a binary alphabet.
We also describe a class of morphisms over multiliteral alphabet for which
the conjecture still holds.
The main result of the article is the following theorem.
\end{extendedabstractversion}

\begin{theorem} \label{th:main}
Let $\varphi$ be a primitive marked morphism and let $\bu$ be its fixed point with finite defect.
If all complete return words of all letters in $\bu$ are palindromes or there exists a conjugate of $\varphi$ distinct from $\varphi$ itself, then $D(\bu) = 0$.
\end{theorem}

Observe that in the case of primitive marked morphisms, as it was noted in
\cite[Cor. 30, Cor.  32]{LaPe14}, there is no ultimately periodic infinite
word $\bu$ fixed point of a primitive marked morphisms such that
$0<D(\bu)<\infty$.

% \todo{Seb ask: why "with card. at least 3"? Edita answers:  You are right,
% cardinality 3 is not necessary.  I forget, that  periodic fixed points of
% marked morphism on binary alphabet have defect 0 as well. }

%As every binary primitive morphism is either marked or its fixed point is primitive, the last theorem has the following corollary for %binary alphabets.

%\begin{corollary} \label{co:main}
%Let $\varphi$ be a binary morphism and $\bu$ its fixed point.
%If $\bu$ is not periodic, then $D(\bu) = 0$ or $D(\bu) = +\infty$.
%\end{corollary}

\begin{extendedabstractversion}
The main ingredients for the presented proof of \Cref{th:main} are the following:
\begin{enumerate}
\item description of bilateral multiplicities of factors in a word with finite palindromic defect (\cite{BaPeSta2}),
\item description of the form of marked morphisms with their fixed points containing infinitely many palindromic factors (\cite{LaPe14}).
\item  observation that non-zero palindromic defect  of a word  implies an
    existence of a factor with a specific property, see \Cref{th:tuesday} for
    the multiliteral case.
\end{enumerate}
\end{extendedabstractversion}

\begin{longversion}
The main ingredients for the presented proofs of \Cref{th:main} and \Cref{th:ZDCbinary} are the following:
\begin{enumerate}
\item description of bilateral multiplicities of factors in a word with finite palindromic defect (\cite{BaPeSta2}),
\item description of the form of marked morphisms with their fixed points containing infinitely many palindromic factors (\cite{LaPe14}).
\item  observation that non-zero palindromic defect  of a word  implies an  existence of a factor with a specific property, see \Cref{binarQ}  for the binary case and \Cref{th:tuesday} for the multiliteral case.
\end{enumerate}
\end{longversion}

\begin{longversion}
The paper is organized as follows:
First we recall notions from combinatorics on words and we list known results that we use in the sequel.
In Section 3, the properties of marked morphisms are studied.
In Section 4, we introduce the notion of a graph of a factor and we interpret bilateral multiplicity of factors in the language of graph theory.
Section 5 is focused on properties of a graph of a factor in the case of language having finite palindromic defect.
The validity of the Zero Defect Conjecture on binary alphabet is demonstrated in \Cref{sec:binary} (\Cref{th:ZDCbinary}).
\Cref{sec:multimarked} contains the proof of \Cref{th:main}.
Comments on counterexamples to two conjectures concerning palindromes form the last Section 8.
\end{longversion}

\section{Preliminaries}

An {\it alphabet} $\A$ is a finite set of symbols called {\it letters}.
A {\it finite word} $w = w_0w_1 \cdots w_{n-1}$ is a finite sequence over $\A$, i.e., $w_i \in \A$.
The {\it length} of $w$ is $n$ and is denoted by $|w|$.
An {\it infinite word} is an infinite sequence over $\A$.
Given words $p,f,s$ with $p$ and $f$ being finite such that $w = pfs$, we say that $p$ is a {\it prefix} of $w$, $f$ is a {\it factor} of $w$, and $s$ is a {\it suffix} of $w$.

\subsection{Language of an infinite word}

 Consider an infinite word $\bu = (u_n)_{n\in \mathbb{N}}$ over the alphabet $\A$.
An index $i\in \mathbb{N}$ is an occurrence of a factor $w = w_0w_1 \cdots w_{n-1}$ of  $\bu$  if
$u_iu_{i+1}\cdots u_{i+n-1} = w$, in other words $w$ is prefix of the infinite word $u_iu_{i+1}u_{i+2}\cdots $.
The set of all factors of $\bu$ is referred to as the {\it language} of $\bu$ and denoted $\Lu$.
The mapping $\C(n)$, defined by $\C(n) = \# \Lu \cap \A^n$, is the {\it factor complexity} of $\bu$.
A word $\bu$ is called recurrent if any factor $w\in \Lu$ has infinitely many occurrences.
If $ i <j$  are two consecutive occurrences  of the factors $w$, then  the factor $u_iu_{i+i} \cdots u_{j}u_{j+1}\cdots u_{j+n-1}$ is
the {\it complete return word} to $w$ in $\bu$.  If any factor of a recurrent word $\bu$  has only finitely many complete return words, then $\bu$ is called {\it uniformly recurrent}.

 {\it Reversal} of  a finite word $w = w_0w_1 \cdots,w_{n-1}$  is the word $\widetilde{w} =  w_{n-1}w_{n-2} \cdots w_{0}$.
A word $w$ is a {\it palindrome} if $w =\widetilde{w} $.
The language of $\bu$ is said to be {\it closed  under  reversal}  if $w \in \Lu$ implies $\widetilde{w} \in \Lu$; $\bu$ is said to be {\it palindromic}  if $\Lu$ contains infinitely many palindromes.
If a uniformly recurrent word $\bu$  is palindromic, then its language is closed under reversal.
The mapping counting the palindromes of length $n$ in $\Lu$ is the {\it palindromic complexity} and is denoted by $\P(n)$, i.e., we have $\P(n) = \# \{ w \in \Lu \colon |w| = n \text{ and } w = \til{w} \}$.

A letter $b\in \A $ is called {\it right } (resp. {\it left}) {\it extension} of $w$ in $\Lu$ if $wb$ (resp. $bw$) belongs to $\Lu$.
In a recurrent word $\bu$ any factor has at least one right and at least one left extension.
A factor  $w$ is {\it right special}  (resp. {\it left special}) if it has    more than one right (resp. left) extension.
A factor $w$ which is simultaneously left and right special is {\it bispecial}.
To describe one-sided and   both-sided  extensions  of a factor $w$ we put
\[  \  E^+(w)    = \{b \in \A \colon  wb \in \Lu \}, \quad   E^-(w)    = \{a \in \A \colon  aw \in \Lu \},
\]
\[
 \text{and} \quad E(w)   = \{(a,b) \in \A^2  \colon awb \in \Lu \}.
\]
The \emph{bilateral multiplicity} $m(w)$ of a factor $w\in\Lu$ is defined as
\[
    m(w) = \#E(w) - \#E^+(w) - \#E^-(w) + 1.
\]
Under the assumption of recurrent language, the second difference of the factor complexity may be expressed using bilateral multiplicities as follows:
\begin{equation} \label{eq:FaCoBiMu}
\Delta^2 \C(n) = \C(n+2) - 2 \C(n+1) + \C(n)
    = \sum_{ \substack{w \in \Lu \\ |w| =  n \\ w \text{ is bispecial}}} m(w)
    = \sum_{ \substack{w \in \Lu \\ |w| =  n}} m(w).
\end{equation}
(See \cite{MR2759107}, Section 4.5.2 for the equation \eqref{eq:FaCoBiMu} and Section 4 for a recent reference on factor complexity in general.)
Note that the last equality in \eqref{eq:FaCoBiMu} follows from the fact that $m(w)$ is nonzero only for bispecial factors in the case of a recurrent language.

A bispecial factor $w\in\Lu$ is said to be
\emph{strong} if $m(w)>0$,
\emph{weak} if $m(w)<0$
and
\emph{neutral} if $m(w)=0$.

\subsection{Palindromic defect}\label{DefDefect}

%Prop 3 in \cite{DrJuPi}
As shown in \cite{DrJuPi} finite words  with zero defect can be characterized  using palindromic suffixes.
More specifically, a word
$w=w_0w_1\cdots w_{n-1}$ has defect $0$ if and only if for any $i =0,1,\ldots, n-1$ the longest palindromic suffix of   $w_0w_1\cdots w_{i}$
is unioccurrent in $w$.
To illustrate this important property,  consider the words
\[
w^{(1)}  = 010010100 \quad \text{ and } \quad w^{(2)}  = 011010011.
\]
mentioned in Introduction.
The longest palindromic suffix of $w^{(1)} $ is  $0010100$  and it is unioccurrent in $w^{(1)} $, whereas
 the longest palindromic suffix of $w^{(2)} $ is $11$ and occurs in $w^{(2)} $ twice.
The index $i$  for which the longest palindromic suffix is not unioccurrent  is called a {\it lacuna} and the number of lacunas equals the palindromic defect of $w$.

Since the number of palindromes in $w$ and in its reversal $\widetilde{w}$ is the same,  we have $D(w) = D(\widetilde{w})$.
Therefore, instead of the longest palindromic suffix one can consider the longest palindromic prefix as well.

The complete return words  were applied in \cite{GlJuWiZa} to characterize infinite words with zero defect.

%Theorem 2.14 in \cite{GlJuWiZa}
\begin{theorem}[\cite{GlJuWiZa}] \label{th:crws_to_pals}
$D(\bu) = 0$ if and only if for all palindromes $w \in \Lu$ all complete return words to $w$ in $\bu$ are palindromes.
\end{theorem}

Before stating a generalization of the previous result we need a new notion.
\begin{definition}
Let $\bu \in \A^\N$ %have its language closed under reversal
and $w \in \Lu$.
A word $c = c_1 c_2 \cdots c_n \in \Lu$ is a \emph{complete mirror return} to $w$ in $\bu$ if
\begin{enumerate}
\item neither $w$ nor $\til{w}$ is a factor of $c_2 \cdots c_{n-1}$, and
\item either $w$ is a prefix of $c$ and $\til{w}$ is suffix of $c$, or $\til{w}$ is a prefix of $c$ and $w$ is a suffix of $c$.
\end{enumerate}
\end{definition}

Note that $c$ is a complete mirror return to $w$ if and only if it is a complete mirror return to  $\til{w}$.

\begin{example} We illustrate  the notion of complete mirror return word on the Fibonacci word  $\bu_F$.  The factors $r_1$, $r_2$  and $r_3$  are  complete mirror returns to $w_1 = 0101$,   $w_2 = 001$ and  $w_3 = 00$ respectively.

$$
\bu_{F} = 010\underbrace{01010}_{r_1}0100101\underbrace{0010100}_{r_2}1\underbrace{0010100}_{r_3}10 \cdots
$$
Note that if $w = \til{w}$, then the complete mirror return words of $w$ and $\til{w}$ coincide with complete return words of $w$.
\end{example}
Using the notion of complete mirror return word  we can reformulate Proposition 2.3 from \cite{BuLuGlZa}.
% M. Bucci, A. de Luca, A. Glem, L. Zamboni: A connection  between palindromic
% and factor complexity using return words...

\begin{proposition}[\cite{BuLuGlZa}]\label{prop:rich}
  Let $\bu \in \A^\N$. We have $D(\bu) = 0$ if and only if for each factor $w \in \Lu$  any complete mirror return word to  $w$  in $\bu$ is a palindrome.
\end{proposition}

A generalization  of the previous statement to finite defect follows from \cite[Cor.~5 and Lemma~14]{BaPeSta5}.

\begin{theorem}[\cite{BaPeSta5}] \label{th:mrs_are_palindromes}
Let $\bu\in\A^\N$ be aperiodic and have its language closed under reversal.
    $D(\bu) < + \infty$ if and only if there exists a positive integer $K$ such that for every factor $w$ of length at least $K$ the occurrences of $w$ and $\til{w}$ alternate and every complete mirror return to $w$  in $\bu$ is a palindrome.
\end{theorem}

% The next important corollary of the previous theorem was proved in \cite{BaPeSta2}. It uses
% the both-sided symmetric extensions of a factor that we note
% \[
% E^=(w)  = \{a \in \A \colon awa \in \Lu \}.
% \] % is the set of the palindromic extensions.

% \begin{theorem}[\cite{BaPeSta2}] \label{th:finite_defect_bs_multiplicities}
% Let $\bu\in\A^\N$ have its language closed under reversal.
% If $D(\bu) < +\infty$, then there exists an integer $K$ such that each bispecial factor $q \in \Lu$ with $|q| \geq K$ satisfies
% \[
% m(q) = \begin{cases} 0 & \text{ if } q \neq \til{q}, \\ \# E^=(q)-1 & \text{ if } q = \til{q}. \end{cases}
% \]
% \end{theorem}

% \begin{proof}
% It follows from Theorem 2 and Proposition 6 of \cite{BaPeSta5} that there exists an integer $L$ such that for all $n \geq L$ we have
% $$
% \Delta \C(n) + 2 = \P(n+1) = \P(n)
% $$
% where $\C(n)$ is the factor complexity of $\bu$ and $\P(n)$ is the palindromic complexity of $\bu$.

% \end{proof}

\subsection{Morphisms}

In this section we concentrate  on  primitive morphisms.
For a morphism $\varphi: \A^* \to \A^*$ consider the maps $\Fst(\varphi), \Lst(\varphi) : \A \to \A$  defined  by the formula
\[
\Fst(\varphi)(a) = \text{the first letter of } \varphi(a) \qquad \text{ and } \qquad  \Lst(\varphi)(a) = \text{the last letter of } \varphi(a)
\]
for all $a \in \A$.
A morphism  $\varphi$ may have more fixed points, see for example the Thue--Morse morphism.
The number  of fixed points of a primitive morphism $\varphi$ is the number of letters for which  $ \Fst(\varphi)(a) =a$.
It is easy to see  that the languages of all fixed points of a primitive morphism coincide and therefore all its fixed points have the same defect.

Recall from Lothaire \cite{Lo2} (Section 2.3.4) that a morphism $\psi$ is a \emph{left
conjugate}  of   $\varphi$,
or that $\varphi$ is a \emph{right
conjugate} of $\psi$, denoted $\psi\triangleright\varphi$, if there
exists $w \in \A^*$ such that
\begin{equation}
\varphi(x)w = w\psi(x), \quad \textrm{for all words } x \in \A^*, \label{FirstCond}
\end{equation}
or equivalently that
$\varphi(a)w = w\psi(a)$, for all letters $a \in \A$.
We say that the word $w$ is the \emph{conjugate word of the relation
$\psi\triangleright\varphi$}.
If,  moreover,  the map $\Fst(\psi)$ is not constant, then $\psi$ is the \emph{leftmost conjugate} of $\varphi$.
Analogously, if $\Lst(\varphi)$ is not constant, then $\varphi$ is the \emph{rightmost conjugate} of $\psi$.

\begin{example}
Let
\[
\varphi: \begin{array}{ll}a &\mapsto abab \\ b &\mapsto abb \end{array} \quad \text{ and } \quad \psi: \begin{array}{ll}a &\mapsto baba \\ b &\mapsto bba \end{array}.
\]
We have $\psi\triangleright\varphi$ and the conjugate word of the relation is $w = a$.
The leftmost conjugate of $\varphi$ (and of $\psi$) is the morphism
\[
a \mapsto abab \quad \text{ and } \quad b \mapsto bab.
\]
\end{example}

If $\varphi$ is a primitive morphism, then any of its (left or right) conjugate  is primitive as well and the languages of their fixed points coincide.

A morphism $\varphi:\A^*\to\A^*$ is \emph{cyclic} \cite{Lo83}
if there exists a word $w\in\A^*$ such that
$\varphi(a)\in w^*$ for all $a\in\A$.
Otherwise, we say that $\varphi$ is \emph{acyclic}.
If $\varphi$ is cyclic, then the fixed point of $\varphi$ is $wwww\ldots$ and
is periodic.
Remark that the converse does not hold.
For example, the morphism determined by $a\mapsto aba$ and $b\mapsto bab$ is
acyclic but its fixed point is periodic.
Obviously,
a morphism is cyclic if and only if it is conjugate to itself with a nonempty
conjugate word.
If a morphism is acyclic, it has a leftmost and a rightmost conjugate.
See \cite{LaPe14} for a proof of these statements on cyclic morphisms.

\section{Marked morphisms}
A morphism $\varphi$  over binary alphabet has a trivial but important property:   the leftmost conjugate of $\varphi$ maps both letters to words with a distinct starting letter and analogously for the rightmost conjugate.
The notion of {\it marked morphism}  extends  this important property  to any alphabet.

\begin{definition}%[marked, well-marked]
\label{def:marked}
Let $\varphi$ be an acyclic morphism.
We say that $\varphi$ is \emph{marked} if
\begin{center}
$\Fst(\varphi_L)$ and $\Lst(\varphi_R)$ are injective
\end{center}
and that $\varphi$ is \emph{well-marked} if
\begin{center}
it is marked and if $\Fst(\varphi_L)=\Lst(\varphi_R)$
\end{center}
where $\varphi_L$ (resp. $\varphi_R$) is the leftmost (resp. rightmost)
conjugate of~$\varphi$.
\end{definition}

\begin{remark} Any injective mapping  $f$ on a finite set  is a permutation
    and thus there exists a power $k$ such that $f^k$ is the identity map.  It
    implies that for any marked morphism $\varphi$ there exists a power $k$
    such that $\varphi^k$  is well-marked and moreover
    $\Fst(\varphi^k_L)=\Lst(\varphi^k_R) = {\rm Id}$.
\end{remark}

\begin{longversion}

% Let $\varphi$ be a well-marked morphism with  $w$ being the word that
% $$\varphi_R(a)w= w\varphi_L(a) \text{ for all } \  a \in  \A\,.$$

\begin{theorem}[\cite{LaPe14}] \label{th:conjugate_palindromes}
Let $\varphi$ be a primitive well-marked morphism and $\bu$ be its palindromic fixed point.
The conjugacy word $w$ of $\varphi_L\triangleright \varphi_R$ is a palindrome and $$\til{
\varphi_R(a)} = \varphi_L(a) \ \text{ for all }\ a \in \A\,.$$
\end{theorem}

%\todo[inline,color=red]{In the sequel, I  assume that $\varphi$ is primitive.  Since to show some properties  of $\Phi$ we need to %apply the previous theorem. }

We are interested in the defect of fixed points of  primitive marked morphisms.
We can consider, instead of the marked morphism  $\varphi$, a suitable power of $\varphi$.
Thus, without loss of generality  we assume that $\varphi$ is well marked and  that $\Fst(\varphi_L)=\Lst(\varphi_R) = {\rm Id}$.  For such $\varphi$ with the conjugacy word $w$ of $\varphi_L \triangleright \varphi_R$ we define the mapping $\Phi : \A^* \to \A^*$ by
$$
 \Phi(u) = \varphi_R(u)w \ \ \text{for all  } \  u \in \A^*\,.
$$

As $\varphi$ is primitive, each of its powers and also each  of its conjugates
have the same language.  Moreover, if we assume that $\bu$  is palindromic,
we  can deduce  using \cite[Lemma 15, Lemma 27, Prop. 28]{LaPe14}
remarkable  properties  of the mapping $\Phi$.

% \todo[backgroundcolor=green,inline]{We say it is in \cite{LaPe14}, should we also say why we redo a proof?}
% \todo[inline]{Seb: Because the proof is spreadout into 3 distincts lemmas and proposition. It is good to have it all in one statement, one place I think.}

\begin{lemma}[\rm\cite{LaPe14}] \label{lem:Phi5properties}
Let $\bu\in\A^\N$ and $u \in \A^*$.
If  $\varphi$  satisfies assumptions of  Theorem \ref{th:conjugate_palindromes},
we have
\begin{enumerate}[\rm(I)]
\item If $u \in \Lu$, then $\Phi(u) \in \Lu$.
\item \label{item:two} $\til{\Phi(u)} = \Phi(\til{u})$.
\item The word $u$ is a palindrome if and only if $\Phi(u)$ is a palindrome.
\item \label{prop:extensions} For any $a,b \in \A$, $aub \in \Lu$ implies $a \Phi(u) b \in \Lu$.
\item If $u$ is a palindromic (respectively non-palindromic) bispecial factor,
    then $\Phi(u)$ is a palindromic (respectively non-palindromic) bispecial factor.
\end{enumerate}
\end{lemma}

\begin{proof}
(I)
Let us find $v$ such that  $uv \in \Lu$ with $|\varphi_L(v)| \geq w.$
We have
\[
\varphi_R(uv)w = \varphi_R(u)w\varphi_L(v).
\]
Since $\varphi_R(uv) \in \Lu$, by erasing a suffix of length greater than or equal to $|w|$ from  $ \varphi_R(u)w\varphi_L(v)$ we obtain a factor of $\Lu$, in particular    $\varphi_R(u)w \in \Lu$.

(II)
Let $u = u_1u_2 \cdots u_n$ with $u_i \in \A$.
We obtain
\[
\Phi(u)
= w \varphi_L(u_1) \cdots \varphi_L(u_n)
= \varphi_R(u_1) \cdots \varphi_R(u_n) w.
\]
Using \Cref{th:conjugate_palindromes} we obtain
\[
\til{\Phi(u)}
 = \til{w} \til{\varphi_R(u_n)} \cdots \til{\varphi_R(u_1)}
 = w \varphi_L(u_n) \cdots \varphi_L(u_1)
 = \Phi(\til{u}).
\]

(III)
Let us note that any marked morphism is injective and thus $\Phi$ is injective as well.
    If $u$ is a palindrome, then $\til{\Phi(u)}=\Phi(\til{u})=\Phi(u)$
    from \Cref{item:two}, therefore $\Phi(u)$ is a palindrome.
    Conversely, if $\Phi(u)$ is a palindrome, then
    $\Phi(u)=\til{\Phi(u)}=\Phi(\til{u})$.
    As $\varphi_L$ is injective, $\Phi$ is injective and the claim follows.

(IV)
Let $aub \in \Lu$.
We have $\Phi(aub) \in \Lu$ and
\[
\Phi(aub) = \varphi_R(a) \varphi_R(u) w_\varphi \varphi_L(b) = \varphi_R(a) \Phi(u) \varphi_L(b).
\]
By our assumption,  $\Lst(\varphi_R)(c) = \Fst(\varphi_L)(c) = {\rm Id}(c)=c$ for any $c\in \A$.
Thus, $a \Phi(u) b$ is a factor $\Phi(aub) \in \Lu$.

(V)
The statement follows from the previous properties.
\end{proof}

\section{Extension graphs of a factor}

To study the Zero Defect Conjecture on a multiliteral alphabet,
we assign graphs to palindromic and non-palin\-dromic  bispecial factors.
These graphs were used already in \cite{BaPeSta2} where only words with zero defect are considered.
These graphs enable to represent
extensions of a bispecial factor and to determine factor complexity, see \cite[p.234--235]{MR2759107}.
They also appear in a more general context in
\cite{BeClDoLePeReRi}. We use these graphs to demonstrate that the definition
of bilateral multiplicity of bispecial factors is related to basic notions of
graph theory which we use later in the proofs.

%Therefore, if $\{(a,-1),(b,+1)\}$ is an edge of $\Gamma(w)$
%then it is also an edge of $\Gamma(\Phi^n(w))$
%and if $\{a,b\}$ is an edge of $\Theta(w)$
%then it is also an edge of $\Theta(\Phi^n(w))$.
%\end{lemma}

\begin{definition}[$\mathbf{\Gamma(w)}$]
Let $\bu\in\A^\N$.
We assign to a factor $w\in\L(\bu)$ the bipartite \emph{extension graph}
$\Gamma(w)=(V,U)$ whose
vertices $V$ consist of the disjoint union of $E^-(w)$ and $E^+(w)$
\[
    V=\left(E^-(w)\times\{-1\}\right) \cup
      \left(E^+(w)\times\{+1\}\right)
\]
and whose edges $U$ are essentially the elements of $E(w)$:
\[
    U = \left\{\left\{ (a,-1), (b,+1)\right\}\colon (a,b)\in E(w)\right\}.
\]
\end{definition}

\begin{lemma}\label{lem:gammagraphtheory}
If $\Gamma(w)$ is connected, then $m(w)\geq 0$ and
\begin{itemize}
\item $m(w)>0$ if and only if $\Gamma(w)$ contains a cycle,
\item $m(w)=0$ if and only if $\Gamma(w)$ is a tree.
\end{itemize}
\end{lemma}

\begin{proof}
    Let $G=(V,U)$ be a graph with vertices $V$ and edges $U$.
    If $G$ is connected then $\#U-\#V+1\geq 0$.
    A connected graph $G=(V,U)$ is a tree if and only if $\#U-\#V\#+1=0$ and it
    contains a cycle if and only if $\#U-\#V+1>0$.
    In the case of the graph $\Gamma(w)$, it turns out that
    \[
	\#U-\#V+1 = \#E(w) - \#E^-(w) -\#E^+(w) + 1 = m(w).\qedhere
    \]
\end{proof}

Another graph will be useful in the case when $w=\til{w}$ and when the
language $\Lu$ is closed under reversal. These two hypotheses imply that
$E^-(w)=E^+(w)$ and that $E(w)$ is symmetric, i.e. $(a,b) \in E(w)$ if and only if $(b,a) \in E(w)$.
\begin{definition}[$\mathbf{\Theta(w)}$]
Let $\bu\in\A^\N$ having its language closed under reversal.
To a palindromic factor $w\in\L(\bu)$ we assign a graph $\Theta(w)=(V,U)$ whose
vertices $V=E^-(w)=E^+(w)$ are exactly the right (or left) extensions of $w$ and
whose edges $U$ are unordered pairs of distinct elements of $E(w)$:
\[
    U = \left\{\left\{ a, b\right\} \colon (a,b)\in E(w), a\neq b\right\}.
\]
\end{definition}
In particular, $\Theta(w)$ does not contain loops.

The next lemma uses the both-sided symmetric extensions of a factor $w$ which are denoted by
\[
E^=(w)  = \{a \in \A \colon awa \in \Lu \}.
\] % is the set of the palindromic extensions.

\begin{lemma}\label{lem:thetagraphtheory}
Suppose that the language $\Lu$ is closed under reversal and $w=\til{w}$.
If $\Theta(w)$ is connected, then $m(w)\geq \#E^=(w)-1$ and
\begin{itemize}
    \item  $m(w)>\#E^=(w)-1$ if and only if $\Theta(w)$ contains a cycle,
    \item  $m(w)=\#E^=(w)-1$ if and only if $\Theta(w)$ is a tree.
\end{itemize}
\end{lemma}

\begin{proof}
    Using the same argument as for the previous lemma, we compute that
    \[
	\#U = \frac{1}{2}\left(\#E(w)-\#E^=(w)\right)
	\quad\text{and}\quad
	\#V = \#E^-(w) = \#E^+(w).
    \]
    Therefore,
    \begin{align*}
	\#U-\#V+1 &= \frac{1}{2}\left(\#E(w)-\#E^=(w)
                    - \#E^-(w) - \#E^+(w)\right)  +1\\
            &= \frac{1}{2}\left(m(w)-\#E^=(w) + 1\right)\qedhere
    \end{align*}
\end{proof}

\begin{example}  \label{example:bucci-vaslet}
    Let  $\bu$ be   the fixed point of the substitution
$\eta:a \mapsto aabcacba, b \mapsto aa, c \mapsto a$ used by Bucci and Vaslet.
The list of all factors of length 2 is:
\[
aa, ab,ac, ba,ca, bc,cb.
\]
The list of all factors of length 3 is:
\[
aaa, aab, abc,acb, baa, bca, cac, cba.
\]
This allows to construct the graphs  $\Theta(w)$  and $\Gamma(w)$ for
$w\in\{\varepsilon,a,b,c\}$ (see Fig.~\ref{fig:examplesgraph}) and the
following table of values for the bilateral multiplicity:
\[
    \begin{array}{c|ccccc}
	  w  & \varepsilon & a & b & c\\
	\hline
	m(w) & 2 & -1 & -1 & -1\\
  \#E^=(w)-1 & 0 & 0  & -1 & -1
    \end{array}.
\]

\begin{figure}[h!]
\begin{center}
\includegraphics[width=0.95\linewidth]{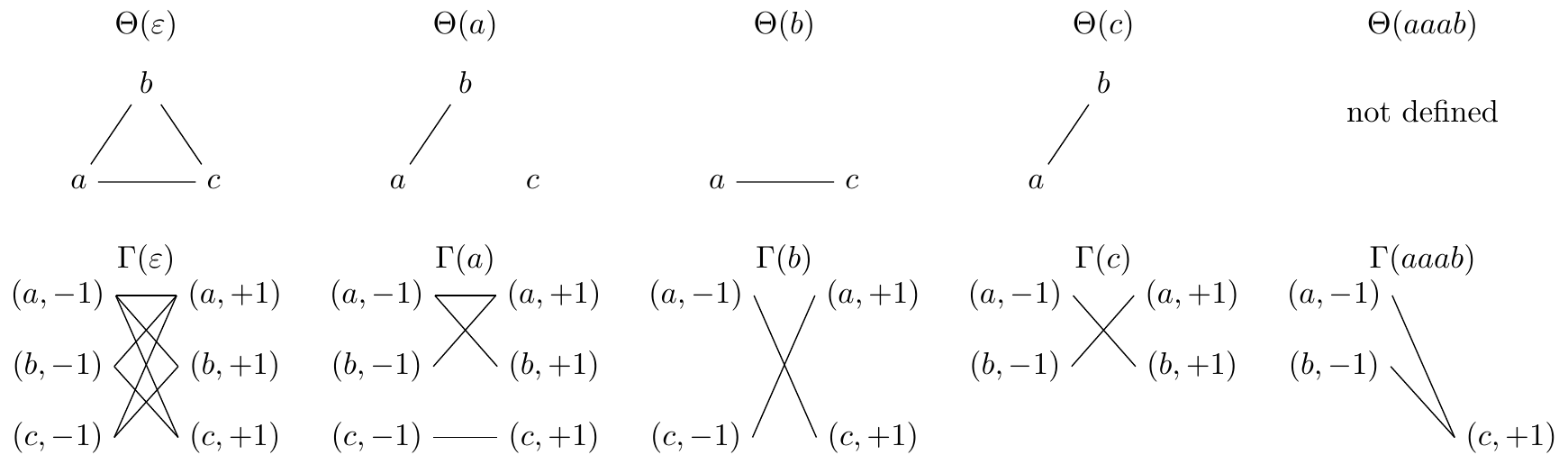}
\end{center}
\caption{Example of graphs $\Theta(w)$ and $\Gamma(w)$  for
$w\in\{\varepsilon,a,b,c,aaab\}$ in the language of
the fixed point of the morphism
$a \mapsto aabcacba, b \mapsto aa, c \mapsto a$.}
\label{fig:examplesgraph}
\end{figure}

\begin{enumerate}
\item The graph $\Theta(\varepsilon)$  has vertices  $V=\{a,b,c\}$ and edges
$U=\bigl\{ \{a,b\}, \{a,c\}, \{b,c\}  \bigr\}$. The graph $\Theta(\varepsilon)$
contains a cycle. The bilateral multiplicity  equals $m(\varepsilon) =2 > 0
=\#E^=(\varepsilon)-1$.

\item The graph $\Theta(a)$  has vertices $V=\{a,b,c\}$ and edges $U=\bigl\{
\{a,b\} \bigr\}$. The graph $\Theta(a)$ is not connected.  The bilateral
multiplicity  equals $m(a) =-1  <  0=\#E^=(a)-1$.

\item The graph $\Theta(b)$  has vertices $V=\{a,c\}$ and edges $U=\bigl\{
\{a,c\} \bigr\}$. The graph $\Theta(b)$ is a tree.  The bilateral multiplicity
equals $m(b) =-1  =\#E^=(b)-1$.
\end{enumerate}
It is easy to see that the graph $\Theta(c)$ is isomorphic to $\Theta(b)$.
The construction of the graphs $\Gamma(w)$ is analogous.
From the extension set $E(aaab)=\{(a,c), (b,c)\}$ of
the non-palindromic left special word $w=aaab$, the graph
$\Gamma(aaab)$ can be constructed (see Fig.~\ref{fig:examplesgraph}). Notice
that it is a tree.

\end{example}

\begin{longversion}

\section{Words with finite palindromic defect}

The graphs introduced in the previous section allow to interpret the palindromic defect in terms of graph theory.
In this section we focus on properties of graphs of a factor under the assumption of finite palindromic defect (\Cref{th:finite_defect_bs_multiplicities} and \Cref{cor:wednesday}).
In \Cref{sec:multimarked}, we study properties of a graph of a factor under the assumption of positive palindromic defect (\Cref{th:tuesday}).

The proof of the main results of this section, namely \Cref{th:finite_defect_bs_multiplicities}, can be excerpted
from \cite[proof of Theorem 3.10]{BaPeSta2}.
However, the mentioned theorem  has a stronger assumption (the palindromic defect of ${\bf u}$ is zero) and it is not stated in terms of graphs as done below in \Cref{cor:wednesday}.
Therefore,  \Cref{th:finite_defect_bs_multiplicities} is accompanied here with an independent proof.
The proof requires the next two lemmas, which  explain the link between complete mirror return word to a
factor $w$ and the connectedness of its associated graphs, and a proposition on the relation of factor and palindromic complexity in a word having finite palindromic defect.

%In particular, \Cref{cor:wednesday} is a translation of known results presented
%in the preliminaries using \Cref{lem:gammagraphtheory,lem:thetagraphtheory}.

\begin{lemma}\label{lem:is-an-edge}
Let $\bu\in\A^\N$ have its language closed under reversal. Suppose that $v$ is a
palindromic complete mirror return word to $w\in\L(\bu)$ such that $b\til{w}$
is a suffix of $v$ and $av\in\L(\bu)$ for some letters $a,b\in\A$. Then
$\{(a,-1),(b,+1)\}$ is an edge of the graph $\Gamma(w)$. If $w$ is a
palindrome and $a\neq b$, then $\{a,b\}$ is an edge of the graph $\Theta(w)$.
\end{lemma}

\begin{proof}
    Let $s\in\A^*$ such that $v=sb\til{w}$. Since $v$ is a palindrome,
    we get $v=wb\til{s}$. Therefore, $awb\in\L(\bu)$ being a
    prefix of $av$ and $(a,b)\in E(w)$. We conclude that $\{(a,-1),(b,+1)\}$
    is an edge of the graph $\Gamma(w)$. Also if $w=\til{w}$ and $a\neq b$,
    we conclude that $\{a,b\}$ is an edge of the graph $\Theta(w)$.
\end{proof}

\begin{lemma}\label{lem:is-connected}
    Let $\bu\in\A^\N$ have its language closed under reversal, $w\in\L(\bu)$ and suppose
    that occurrences of $w$ and $\til{w}$ alternate in $\bu$.
    Suppose that all complete mirror return words to $w$ are palindromes.
    Then $\Gamma(w)$ is connected.
    If $w$ is a palindrome, then $\Theta(w)$ is connected.
\end{lemma}

\begin{proof}
    It suffices to show that there is a path from
    any vertex $(a,-1)$ to any vertex $(b,+1)$ in $\Gamma(w)$.
    Let $(a,-1)$ and $(b,+1)$ be two vertices of $\Gamma(w)$.
    Then $aw,wb\in\Lu$.
    Let $v\in\Lu$ be such that $aw$ is a prefix of $av$ and $b\til{w}$ is a
    suffix of $av$.
    If there are no other occurrences of factors of $\A w\cup\A\til{w}$ in $v$,
    then $\{(a,-1),(b,+1)\}$ is an edge of the graph $\Gamma(w)$ from
    \Cref{lem:is-an-edge}.
    Suppose that
    \[
	a_1w,\, b_1\til{w},\,
	a_2w,\, b_2\til{w},\,
	\ldots\,,
	a_nw,\, b_n\til{w}
    \]
    are consecutive occurrences of factors of $\A w\cup\A\til{w}$ in $v$ where
    $a=a_1$, $b=b_n$ and $n\geq 2$.
    From \Cref{lem:is-an-edge}, $\{(a_i,-1),(b_i,+1)\}$ is an edge of the graph
    $\Gamma(w)$ for all $i$ with $1\leq i\leq n$. Also,
    $\{(a_{i+1},-1),(b_i,+1)\}$ is an edge of the graph
    $\Gamma(w)$ for all $i$ with $1\leq i\leq n-1$.
    Therefore, we conclude that there exists a path from
    $(a,-1)$ to $(b,+1)$.

    Assume $w=\til{w}$.
    Let $a,b\in E^-(w)=E^+(w)$ be two distinct vertices of
    $\Theta(w)$. Then $aw,bw\in\L(\bu)$.
    We want to show that there exists a path from $a$ to
    $b$ in $\Theta(w)$.
    Among the occurrences of factors in $\A w$, if there exist two
    consecutive occurrences of $aw$ and $bw$, then
    $\{a,b\}$ is an edge of $\Theta(w)$ from
    \Cref{lem:is-an-edge}. Otherwise, we conclude that there exists a path from
    $a$ to $b$ by transitivity.
\end{proof}

\begin{corollary} \label{co:tree_and_multiplicities} %PREVIOUS TH. 7
Let $\bu\in\A^\N$ have its language closed under reversal.
If $D(\bu) < +\infty$, then there exists an integer $K$ such that for each bispecial factor $w \in \Lu$ with $|w| \geq K$ the graph $\Gamma(w)$ is  connected. If $w$ is moreover a palindrome,  then also the graph  $\Theta(w)$ is connected.
\end{corollary}

\begin{proof}
  If $\bu$ is not aperiodic, then the claim is trivially satisfied as there is
only a finite number of bispecial factors.

  If $\bu$ is aperiodic, \Cref{th:mrs_are_palindromes} implies that there exists a positive integer $K$ such that for every factor $w \in \Lu$ longer than $K$, the occurrences of $w$ and $\til{w}$ alternate and every complete mirror return to  $w$  in $\bu$ is a palindrome.
We conclude from \Cref{lem:is-connected} that the graph $\Gamma(w)$ is
connected. Also if $w$ is a palindrome, then $\Theta(w)$ is connected.
\end{proof}

The following claim may be deduced from \cite{BaPeSta5}.

\begin{proposition}[{\cite[Th. 2, Prop. 6]{BaPeSta5}}]  \label{prop:factvspalcomplexity}
Let $\bu\in\A^\N$ have its language closed under reversal.
If $D(\bu) < +\infty$, then there exists an integer $M$ such that for all $n
\geq M$ we have
\[
\Delta^2 \C(n) = \P(n+2) - \P(n).
\]
\end{proposition}

\begin{proof}
    Since $\Lu$ is closed under reversal,
Proposition 6 from \cite{BaPeSta5} says that
there exists an integer $M$ such that for all $n \geq M$ we have
\[
\Delta \C(n) + 2 \geq \P(n+1) + \P(n).
\]
Since $D(\bu)<+\infty$,
Theorem 2 from \cite{BaPeSta5} together with the above inequality implies that there exists an integer $M$ such that for all $n \geq M$ we have
\[
\Delta \C(n) + 2 = \P(n+1) + \P(n).
\]
From this we conclude:
\[
\Delta^2 \C(n) = \Delta \C(n+1) - \Delta \C(n) = \P(n+2) + \P(n+1) - \P(n+1) - \P(n) = \P(n+2) - \P(n).\qedhere
\]
\end{proof}

\begin{theorem} \label{th:finite_defect_bs_multiplicities} %PREVIOUS TH. 7
Let $\bu\in\A^\N$ have its language closed under reversal.
If $D(\bu) < +\infty$, then there exists an integer $K$ such that each bispecial factor $w \in \Lu$ with $|w| \geq K$ satisfies
\[
m(w) = \begin{cases} 0 & \text{ if } w \neq \til{w}, \\ \# E^=(w)-1 & \text{ if } w = \til{w}. \end{cases}
\]
\end{theorem}

\begin{proof}
% Let $\C(n)$ be the factor complexity of $\bu$, i.e., the number of distinct factors of length $n$ in $\Lu$.
% And let $\P(n)$ be the palindromic complexity of $\bu$, i.e., the number of distinct palindromic factors of length $n$ in $\Lu$.
% The second difference of the factor complexity may be expressed using bilateral multiplicities as follows (see for instance Section 4.5.2 in \cite{MR2759107}):
% \begin{equation} \label{eq:FaCoBiMu}
% \Delta^2 \C(n)
%     = \sum_{ \substack{w \in \Lu \\ |w| =  n \\ w \text{ is bispecial}}} m(w)
%     = \sum_{ \substack{w \in \Lu \\ |w| =  n}} m(w)
% \end{equation}
%     Note that $m(w)$ is defined for every factor but is non-zero only for
% bispecial factors $w$ since the language of $\bu$ is closed under reversal, and thus recurrent.
Let $K_1$ be the constant  given by  \Cref{co:tree_and_multiplicities}. If  $w$ is a bispecial factor with  $|w| >K_1$,
we conclude from \Cref{lem:is-connected} that the graph $\Gamma(w)$ is
connected. Also if $w$ is a palindrome, then $\Theta(w)$ is connected.
It follows from \Cref{lem:gammagraphtheory} that $m(w)\geq 0$.
If $w$ is a palindrome, \Cref{lem:thetagraphtheory} implies $m(w)\geq \#E^=(w)-1$.

If $w$ is not a bispecial factor, then $m(w) = 0$ and, moreover, if $w$ is not a bispecial factor and $w=\til{w}$, then by closedness under reversal we have $\#E^=(w) = 1$, and thus $m(w)=0= \#E^=(w)-1$.

Suppose by contradiction that for every integer $N$ there exists a non-palindromic factor $v$ of
length $|v|>N$ such that  $m(v)>0$ or
there exists a palindromic factor $v$ of
length $|v|>N$ such that  $m(v)>\#E^=(v)-1$.
As closedness under reversal implies recurrence, using \eqref{eq:FaCoBiMu} we obtain that for every integer $N$ there exists $n=|v|>N$ such that
\begin{equation} \label{eq:estim}
\Delta^2 \C(n)
  = \sum_{ \substack{w \in \Lu \\ |w| = n \\ w \neq \til{w}}} m(w) + \sum_{ \substack{w \in \Lu \\ |w| = n \\ w = \til{w}}} m(w)
  >  0 + \sum_{ \substack{w \in \Lu \\ |w| = n \\ w = \til{w}}} \left( \#E^=(w)-1 \right) = \P(n+2) - \P(n).
\end{equation}
This contradicts \Cref{prop:factvspalcomplexity} and ends the proof of the theorem with $K = \max\{K_1,M\}$.
\end{proof}

The following result is a direct consequence of
\Cref{lem:gammagraphtheory} and \Cref{lem:thetagraphtheory}.
It allows to interpret the previous theorem in terms of graph theory.

\begin{corollary} \label{cor:wednesday}
Let $\bu\in\A^\N$ be an infinite word with its language closed under reversal and $D(\bu) < + \infty$.
There exists a positive integer $K$ such that
for every $w \in \Lu$ of length at least $K$
\begin{itemize}
\item if $w$ is not a palindrome, then the graph $\Gamma(w)$ is a tree,
\item if $w$ is a palindrome, then the graph $\Theta(w)$ is a tree.
\end{itemize}
\end{corollary}

\begin{proof}
Let $\bu\in\A^\N$ have its language closed under reversal.
From \Cref{co:tree_and_multiplicities}, there exists an integer $K_1$ such that for each bispecial factor $w \in \Lu$ with $|w| \geq K_1$ the graph $\Gamma(w)$ is  connected. If $w$ is moreover a palindrome,  then also the graph  $\Theta(w)$ is connected.

From \Cref{th:finite_defect_bs_multiplicities},
it follows that there exists a constant $K_2$ such that every factor $w$ longer than $K_2$ satisfies
\[
m(w) = \begin{cases} 0 & \text{ if } w \neq \til{w}, \\ \# E^=(w)-1 & \text{ if } w = \til{w}. \end{cases}
\]
Let $K = \max \{ K_1, K_2 \}$ and $w$ be a factor of $\Lu$ such that $|w|>K$.
If $w\neq\til{w}$, \Cref{lem:gammagraphtheory} implies that $\Gamma(w)$ is a
tree. If $w=\til{w}$, \Cref{lem:thetagraphtheory} implies that $\Theta(w)$ is
a tree.
\end{proof}

\end{longversion}

\section{Proof of Zero Defect Conjecture for binary alphabet} \label{sec:binary}

The binary alphabet offers less variability for the construction of a strange phenomenon.
The recent counterexamples to two conjectures  concerning palindromes in fixed points of primitive morphisms ---  namely
the Bucci-Vaslet counterexample to the Zero Defect Conjecture and the Labb\'e counterexample to the Hof-Knill-Simon (HKS) conjecture --- use ternary alphabet.
That conjecture \cite{HoKnSi} asks whether all palindromic fixed points of primitive substitutions are fixed by some conjugate of a morphism of the form $\alpha\mapsto p_\alpha p$ where $p_\alpha$ and $p$ are palindromes.
On a binary alphabet, Tan demonstrated the validity of the HKS conjecture, see \cite{BoTan}.
Here we prove the Zero Defect Conjecture on a binary alphabet.

\begin{lemma}\label{binarQ}
Let $\A =\{0,1\}$ and $\bu \in \A^{\mathbb{N}}$. If  $\Lu$ is closed under reversal and $D(\bu) >0$, then there exists a non-palindromic factor $q \in \Lu$ such that
$0q0, 0q1, 1q0, 1q1 \in \Lu$.

\end{lemma}
\begin{proof}

By \Cref{prop:rich}, as $D(\bu) >0$, there exist factors $v$  and $w$ in $\Lu$ such that $v$ is a complete mirror return word  to $w$  and $v$  is not a palindrome.
Let us consider the shortest   $v$  with this property.
For this  fixed $v$  we find the longest $w$ such that $v$ is a complete mirror return word to  $w$.
It means that $v$ has a prefix $wa$  and a suffix $b\widetilde{w}$ where $a,b \in \A$ and $a\neq b$.
Since on a binary alphabet every complete mirror return word to a letter is always a palindrome, we have $|w| >1$.
Without loss of generality  we can write  $w=0q$ with $q\neq \varepsilon$.
Consequently $v=0u0$.
Clearly $u$ has a prefix $q$, the word $u$ has  a suffix $\widetilde{q}$ and $u$ is not a palindrome.
Our  choice  of $v$  (to be  the shortest non-palindromic mirror return to a factor) implies   that $u$ is not a complete mirror return  word to $q$ and thus $q$ or $\widetilde{q}$ has another occurrence inside $u$.
Since $v$ is a complete mirror return word to $w=0q$,
\begin{equation}\label{help}
\text{  $0q$ and $\widetilde{q}0$ do not occur in $u$. }
\end{equation}
Let us suppose that $q =\widetilde{q}$.
Consider the shortest prefix of $u$ which has exactly two occurrences of $q$.
It is palindrome.
Since $v$ has a prefix $wa = 0qa$ the second occurrence of $q$ is extended to the left as $aq$.
Analogously,   consider  the shortest suffix of $u$ which contains exactly two occurrences of $q$.
It is a palindrome and thus the penultimate occurrence of $q$  is extended to the right as $qb$.
This contradicts   \eqref{help} as $a\neq b$.
We conclude that $q$ is not a  palindrome.

Now we show that  occurrences of $q$ and  $\widetilde{q}$ in $u$ alternate.
Assume that there exists a factor of $u$, denoted by $u'$, such that $q$  is a prefix and a suffix of $u'$ and $u'$ does not contain $\widetilde{q}$.
It follows that the longest palindromic suffix of $u'$ is not unioccurrent in $u'$.
Therefore  $D(u') \geq 1$ (see \Cref{DefDefect}),  which  contradicts the minimality of $|v|$.

The minimality of $|v|$ implies that all mirror return words to $q$ in $u$ are palindromes.
Therefore, the leftmost  occurrence of $\widetilde{q}$   in $u$   is extended to the left  as $a\widetilde{q}$     and the rightmost  occurrence of $q$ in   $u$   is extended to the  right   as $qb$.
From  \eqref{help}  we deduce that $0qa$, $a\widetilde{q}1$, $1qb$, and $b\widetilde{q}0$ belong to $\Lu$.
The assumption that $\Lu$ is closed under reversal and the fact that $a\neq b$ finish the proof.
\end{proof}

\begin{theorem} \label{th:ZDCbinary} Let  $\bu\in\A^\N$ be a fixed point of a
    primitive morphism $\varphi$   over a binary alphabet $\A$. If $D(\bu) < +\infty$,  then $D(\bu) =0$ or $\bu$ is periodic.
\end{theorem}

\begin{proof}
Assume  the contrary, i.e.,  $\bu$ is not periodic  and   $D(\bu)>0$ and let $\A= \{0,1\}$.

Since $D(\bu)$ is finite,  $\bu$  is palindromic.
As $\varphi$ is primitive, $\Lu$ is uniformly recurrent.
Any uniformly recurrent word which is palindromic has its language closed under reversal.
Due to \Cref{binarQ} there exists a strong bispecial non-palindromic factor $q$ with $m(q) =1$.

Since      $\bu$ is not periodic,  the morphism $\varphi$ is acyclic.
On the binary alphabet,  it means  that $\varphi$ is well-marked.
Applying  repeatedly \Cref{lem:Phi5properties} (IV) and (V), we can construct an
infinite sequence  of strong bispecial factors $q,  \Phi(q), \Phi^{2}(q), \Phi^{3}(q), \ldots $,  each with bilateral multiplicity $1$.
By \Cref{lem:Phi5properties} (III), all these bispecial factors are non-palindromic.
This contradicts \Cref{th:finite_defect_bs_multiplicities}.
\end{proof}

\end{longversion}

\begin{longversion}
\section{Proof of Zero Defect Conjecture for marked morphisms} \label{sec:multimarked}

%We have prepared all ingredients needed for proving our main theorem.
At first we have to stress that unlike the binary version,  the statement of Theorem \ref{th:main}  does not speak about periodic fixed points.
The following result from \cite{LaPe14} allows to
deduce that on  a larger  alphabet  there is no ultimately periodic infinite word $\bu$ fixed point of a
primitive marked morphism such that $0<D(\bu)<\infty$.

\begin{proposition}{\rm\cite[Cor. 30, Cor. 32]{LaPe14}}\label{cor:onlyBinaryAlphabet}
Let $\bu$ be an eventually periodic fixed point of a primitive marked morphism $\varphi$
over an alphabet $\mathcal{A}$. If $\bu$ is palindromic, then
$\mathcal{A}=\{0,1\}$ is a binary alphabet and $\bu$ equals $(01)^\omega$ or
$(10)^\omega$.
\end{proposition}

Due to the previous proposition, a fixed point of a marked morphisms on binary alphabet
 is  either not eventually periodic or equal to $(01)^\omega$  or $(10)^\omega$.
Since both words $(01)^\omega$ and $(10)^\omega$ have defect zero and the Zero
Defect Conjecture for binary alphabet is proven by \Cref{th:ZDCbinary},  we
may restrict ourselves to  alphabets with cardinality at least three.

First, we prove a multiliteral analogue of Lemma \ref{binarQ}  for words with its language closed under reversal and with positive
palindromic defect.

\begin{theorem} \label{th:tuesday}
Let $\bu\in\A^\N$ have its language closed under reversal. If $D(\bu) > 0$, then
either
\begin{enumerate}
\item there exists a non-palindrome $q\in \Lu$ such that
    $\Gamma(q)$ contains a cycle or
\item there exists a palindrome $q\in \Lu$ such that
  $\Theta(q)$ contains a cycle.
\end{enumerate}
Moreover, if the empty word is the unique factor $q$ with the above property, then there exists a letter with a non-palindromic complete return word.
\end{theorem}

% \todo[inline]{Seb says: it is not clear where the last sentence of the theorem is
% proven in the proof. It should be done at the end of the proof maybe.} Edita:
% added to the point 2 of the proof.

\begin{proof}
Since $D(\bu) > 0$,  there exists a word $v=v_0v_1 \cdots v_n$ such that $w$ is a prefix of
$v$, $\widetilde{w}$ is a suffix of $v$, $v$ does not contain other
occurrences of $w$ or $\widetilde{w}$, $v$ is not a palindrome and $|w|\geq1$.
Suppose that $v$ is a word of minimal length with this property and suppose that
$w$ is the longest prefix of $v$ such that $\widetilde{w}$ is a suffix of $v$. Then there exist letters  $\alpha\neq \beta$ such that $w\alpha$ is a prefix and $\beta \widetilde{w}$ is a suffix of $v$.    Let us define  $t\in \A$ and $q\in \A^*$  to satisfy
$w=tq$ \  (see Figure~\ref{fig:factor_v}).

\begin{figure}[h]
\begin{center}
\begin{tikzpicture}[scale=0.7]
\draw (0,4) rectangle (10,5) node [midway]{$v$};
\draw (0,3) rectangle (3,4)  node [midway]{$w$};
\draw (3,3) rectangle (4,4)  node [midway]{$\alpha$};
\draw (6,3) rectangle (7,4)  node [midway]{$\beta$};
\draw (7,3) rectangle (10,4) node [midway]{$\widetilde{w}$};
\draw (0,2) rectangle (1,3)  node [midway]{$t$};
\draw (1,2) rectangle (3,3)  node [midway]{$q$};
\draw (7,2) rectangle (9,3)  node [midway]{$\widetilde{q}$};
\draw (9,2) rectangle (10,3)  node [midway]{$t$};
\end{tikzpicture}
\end{center}
\caption{The complete mirror return word $v$ to the factor $w$.}
\label{fig:factor_v}
\end{figure}
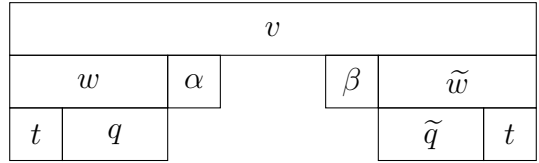

We discuss three cases:

\begin{enumerate}

\item

Let us suppose  $q =\widetilde{q} \neq \varepsilon$.    Due to  the minimality of $v=v_0v_1 \ldots v_n = tq\alpha \cdots \beta qt$, the non-palindromic factor $v_1v_2\ldots v_{n-1} = q\alpha \cdots \beta q$ cannot be a complete return word to $q$ and thus contains at least 3 occurrences of $q$.    Let $k$ be the number of occurrences $q$ in $v$. For $ i=1,2,\ldots, k$,  denote by $ \gamma_i$  the letter which precedes  the  $i^{th}$  occurrence of $q$  and by $\delta_i$ the letter which succeeds the  $i^{th}$  occurrence of $q$.

$\bullet$ Obviously,  $\gamma_1 = t$, $\delta_1 =\alpha$,  and $\gamma_k =\beta$ and  $\delta_k =t$.

$\bullet$ Since $v$ is a complete mirror return word to the factor  $w=tq$, necessarily  $t\neq \gamma_i$  for $i=2,\ldots, k$ and $t\neq \delta_i$  for $i=1,\ldots, k-1$. In particular, $\alpha\neq t$ and $\beta\neq t$.

$\bullet$
  Since each complete return word to $q$ in $v$ is  a palindrome, $\delta_i=\gamma_{i +1}$ for $i = 1, 2,\ldots, k-1$.
We artificially put  $\gamma_{k+1}=\delta_k =t$.

\medskip

According to the definition of $\Theta(q)$,  if $\gamma_i\neq \gamma_{i+1}= \delta_i$, then    the pair $\{\gamma_{i}, \gamma_{i+i}\} $ forms an edge. We want to find a cycle in $\Theta(q)$.
For this purpose we modify  the sequence of letters  $\gamma_1, \gamma_2, \ldots, \gamma_k, \gamma_{k+1}$ as follows:
 If   $\gamma_{j+1}=\gamma_j$ for some index $j=1, \ldots, k$, then  we erase from  the sequence  the ${(j+1)}^{th}$ entry $\gamma_{j+1}$. Then the  modified sequence is a path in $\Theta(q)$ which starts and ends at $t$. The second vertex on the path is $\alpha$, the penultimate vertex is $\beta$. As $\alpha\neq\beta$, the graph $\Theta(q)$ contains a cycle.

\medskip

\item Let us suppose that  $q=\varepsilon$. Now  $v=v_0v_1 \ldots v_n  = t\alpha v_2v_3\cdots \beta t$.
It means that $v$ is a complete return to the letter $t$ which is non-palindromic.
If $v_i\neq v_{i+1}$, the  pair of consecutive letters  $\{v_i,v_{i+1}\}$ is an  edge in  the graph $\Theta(\varepsilon)$  connecting  vertices $v_i$ and $v_{i+1}$. If we erase from the sequence $v_0,v_1,  \ldots ,v_n$  each vertex $v_{j+1}$ which coincides with its predecessor $v_{j}$, we get   a path starting and ending  in  the vertex $t$. The first edge on this path is   $\{t,\alpha\}$, the last one is  $\{t,\beta\}$.   As $\alpha\neq \beta$, the graph  $\Theta(\varepsilon)$ contains a cycle.

\medskip

\item Now we assume that $q \neq \widetilde{q} $.  Note that occurrences of $q$ and $\widetilde{q}$ alternate inside $v$. Indeed,
suppose the contrary, that is there exists a complete return word $z$ of $q$
that has no occurrences of $\widetilde{q}$ and $z$ is a factor of $v$. The
longest palindrome suffix of $z$ must be shorter than $q$. Therefore the longest
palindromic suffix of $z$ is not unioccurrent in $z$. This contradicts the
minimality of $v$. Note also that $v$ must contain other occurrences of $q$ or
$\widetilde{q}$ inside or otherwise we get a contradiction on minimality of $v$. Let us denote $k$ the number of occurrences of $q$ in $v$. Clearly $k$ equals to the number of occurrences of $\widetilde{q} $ as well.

Again we denote by $\gamma_{i}$ the letter which precedes  the  $i^{th}$  occurrence of $q$  and by $\delta_i$ the letter which succeeds the  $i^{th}$  occurrence of $q$.  In particular, $\gamma_1 = t$ and $\delta_1 =\alpha$.
Analogously, we denote by $\tilde{\gamma_{i}}$ the letter which precedes  the  $i^{th}$  occurrence of $\widetilde{q}$  and by $\tilde{\delta_i}$ the letter which succeeds the  $i^{th}$  occurrence of $\widetilde{q}$.  In particular,  $\tilde{\gamma_{k}}=\beta$ and $\tilde{\delta_{k}}=t$. Point out three important facts:

$\bullet$  $\gamma_iq\delta_i \in \Lu$ implies $\{(\gamma_i, -1), (\delta_i, +1)\}$ is an edge in $\Gamma(q)$
for $i=1,2,\ldots, k$.

$\bullet$  As the language $\Lu$ is closed  under reversal,  $\tilde{\gamma_i}\widetilde{q}\tilde{\delta_i} \in \Lu$   implies $\{(\tilde{\delta}_i, -1), (\tilde{\gamma}_i, +1)\}$ is an edge in $\Gamma(q)$
for $i=1,2,\ldots, k$.

$\bullet$   Due to minimality of $v$,  any mirror return to $q$ in $v$ is a palindrome.  Thus $\delta_i=\tilde{\gamma_i}$  for $i =1,2,\ldots,k$ and  $\tilde{\delta}_i=\gamma_{i+1}$ for $i=1,2,\ldots, k-1$.

\medskip

Therefore, $\{(\gamma_i, -1), (\tilde{\gamma}_i, +1)\}$ is an edge in $\Gamma(q)$
for $i=1,2,\ldots, k$,
$\{(\tilde{\gamma}_i, +1), (\gamma_{i+1}, -1)\}$ is an edge in $\Gamma(q)$
for $i=1,2,\ldots, k-1$ and
$\{(\tilde{\gamma}_k, +1), (\tilde{\delta}_k, -1)\}$ is an edge in $\Gamma(q)$.
We can summarize that the sequence of vertices

\centerline{$(\gamma_1, -1), (\tilde{\gamma_1}, +1), (\gamma_2, -1),
(\tilde{\gamma_2}, +1), \ldots,(\gamma_k, -1), (\tilde{\gamma_k},
+1),(\tilde{\delta}_{k}, -1)$}

forms a path in the bipartite graph $\Gamma(q)$ with
$\gamma_1=\tilde{\delta_k}=t$ and
$\tilde{\gamma_1}=\alpha\neq\beta=\tilde{\gamma_k}=t$. In this path the first  and
the last vertices coincide and the second and the penultimate vertices are
distinct. Thus the graph  $\Gamma(q)$ contains a cycle.\qedhere
\end{enumerate}
\end{proof}

As we have seen in \Cref{example:bucci-vaslet} for the fixed point $\bu$ of
the morphism
$\eta:a \mapsto aabcacba, b \mapsto aa, c \mapsto a$ for which the defect is known
to be positive, the graph
$\Theta(\varepsilon)$ contains a cycle.
Since the defect of $\bu$ is finite,
\Cref{cor:wednesday} also applies. Thus there are no arbitrarily large
palindromic factors $w$ containing a cycle in their graph $\Theta(w)$
nor non-palindromic factors $w$ containing a cycle in their graph $\Gamma(w)$.
This is readily seen on the conjugacy word of $\eta_L \triangleright \eta_R$
which is $aaa$ (see Fig.~\ref{fig:examplesgraphaaa}).
\begin{figure}[h!]
\begin{center}
\includegraphics{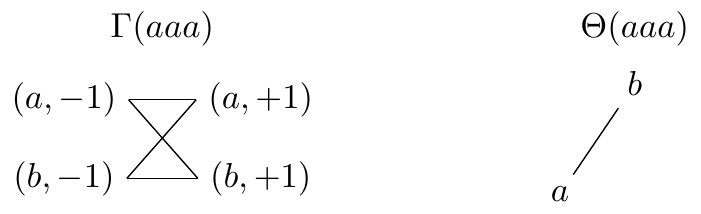}
\end{center}
\caption{$\Gamma(aaa)$ contains a cycle but $\Theta(aaa)$ is a tree
in the language of
the fixed point of the morphism
$a \mapsto aabcacba, b \mapsto aa, c \mapsto a$.}
\label{fig:examplesgraphaaa}
\end{figure}

We are now ready to finish the proof for the multiliteral case.

\begin{proof}[Proof of \Cref{th:main}]
As the languages of the fixed points of $\varphi$ and $\varphi^k$ coincide, we may assume without loss of generality that the marked morphism $\varphi$  has already the property  $\Lst(\varphi_R) = \Fst(\varphi_L)={\rm Id}$.

Proving  that the Zero Defect Conjecture holds in the case of
marked morphisms amounts to prove that the defect is either zero or $+\infty$.
Let us assume on the contrary that $0 < D(\bu) < +\infty$.
It follows that $\bu$ is palindromic.
The primitivity of $\varphi$ implies that $\Lu$ is closed under reversal.

\Cref{th:tuesday} implies that there exists a factor $q$ such that if $q \neq \til{q}$ the graph $\Gamma(q)$ contains a cycle, or if $q=\til{q}$, the graph $\Theta(q)$ contains a cycle.
\Cref{lem:Phi5properties}, property \ref{prop:extensions}, implies that for all $n$, there is a cycle in the graph of $\Phi^n(q)$.

If $q \neq \varepsilon$, then the primitivity of $\varphi$ implies that $\lim_{n \to + \infty} |\Phi^n(q)| = + \infty$.
If $q = \varepsilon$, then, again by \Cref{th:tuesday}, there exists a letter having non-palindromic complete return word.
By the assumption of the theorem, there must exist a conjugate of $\varphi$ distinct from $\varphi$ itself.
It implies that the conjugacy word of $\varphi_L \triangleright \varphi_R$ is nonempty, i.e., $\Phi(\varepsilon) \neq \varepsilon$.
Moreover, $\lim_{n \to + \infty} |\Phi^n(q)| = + \infty$.

To conclude, we have that $\lim_{n \to + \infty} |\Phi^n(q)| = + \infty$ and there is a cycle in the graph of $\Phi^n(q)$ for all $n$.
This is a contradiction with \Cref{cor:wednesday}.
\end{proof}

\section{Comments}

Let us comment two conjectures concerning palindromes in languages of fixed points of primitive morphisms.

\begin{itemize}

\item  The counterexample to the Zero Defect Conjecture  in full generality
    was already mentioned in the Introduction. It is taken  from \cite{BuVa12}.
The fixed point of $$\varphi: a \mapsto aabcacba, b \mapsto aa, c \mapsto a$$ has finite positive palindromic defect and is not periodic.
There is a remarkable property of the fixed point $\bu = \varphi(\bu)$.

Let $\mu: a \mapsto ap, p \mapsto ap aa a ap a aa ap$ be a morphism over the binary alphabet $\{a,p\}$. Let us denote $\bv$ the fixed point of $\mu$.  Then one can easily verify that $\bu = \pi(\bv)$, where   $\pi: a \mapsto a, p \mapsto abcacba$. Moreover,  $\bv$ has zero defect.

In other words, the counterexample word is just an image under $\pi$  of a purely morphic binary word with zero defect.

\item  The counterexample to the question of Hof, Knill and Simon (recalled in~\Cref{sec:binary}) given in
    \cite{La2013}   by the first author is
\[
\psi:a\mapsto aca, b\mapsto cab, c\mapsto b.
\]
%That question asked whether all palindromic fixed points of primitive substitutions are fixed by some conjugate of a morphism of the form $\alpha\mapsto p_\alpha p$ where $p_\alpha$ and $p$ are palindromes.
As mentioned in \cite{MaPeSta1_submitted}, the fixed point $\bu = \psi(\bu)$  is again an image  of  a Sturmian word  $\bv$ under a morphism  $\pi:  \{0,1\} \mapsto \{a,b,c\}$ and the Sturmian word $\bv$ itself  is a fixed point of  a morphism over binary alphabet $\{0,1\}$. Since $\bv$ is Sturmian, its   defect is zero.
\end{itemize}

Both counterexamples are in some sense degenerate. Both words are on ternary alphabet, but the binary alphabet is hidden in their structure.  For further research in this area, it would be instructive to find another kind of counterexamples to  both mentioned conjectures. In this context  we mention that  the  second and third  authors showed in \cite{PeSta_Milano_IJFCS} that any uniformly recurrent infinite word $\bu$ with a finite defect is a morphic  image of a word $\bv$  with defect $0$.

\end{longversion}

\section*{Acknowledgements}

We are grateful to the anonymous referees for their many valuable comments one
of them leading to a reorganisation of the structure of the article including a
proof of \Cref{th:finite_defect_bs_multiplicities} which generalizes Theorem 3.10
proved for rich words in \cite{BaPeSta2}.

The first author is supported by a postdoctoral Marie Curie fellowship
(BeIPD-COFUND) cofunded by the European Commission.
The second and the third authors acknowledge support of GA\v CR 13-03538S (Czech Republic).

%%%%%%%%%%%%%%%%
% Bibliographie %
%%%%%%%%%%%%%%%%%
%\bibliographystyle{plain}
\bibliographystyle{siam}
\IfFileExists{biblio.bib}{\bibliography{biblio}}{\bibliography{../biblio}}

\end{document}